\documentclass[12pt,a4paper,reqno]{amsart}
\usepackage[margin=2.2cm]{geometry}
\usepackage{microtype}
\usepackage[cal=euler]{mathalfa}
\usepackage{enumitem}
\setlist{label=(\arabic*)}
\usepackage{bm}
\usepackage{ytableau}

\input{commands_new}
\defcite\mel{mellit_poincare}
\defcite\melc{mellit_integrality}
\defcite\sch{schiffmann_indecomposable}
\defcite\mac{macdonald_symmetric}
\defcite\gar{garsia_remarkable}
\defcite\msco{mozgovoy_counting} 
\defcite\msc{mozgovoy_countinga}

\def\dt{\hat\Omega}
\def\idt{\Omega}
\def\pdt{\dt^+}
\def\pidt{\idt^+}
\def\cdt{\dt^\circ}
\def\cidt{\idt^\circ}

\opr{hsp}{\cM} 
\opr{hst}{\fM} 
\opr{hgc}{Higgs} 

\def\wr{\alpha} 
\def\vr{\bm\alpha} 
\def\bq{\bm q} 

\def\mm{\widetilde H}
\dmat\lrg{\Ring_\la} 
\opr{lrg}{\Ring_\la} 

\opr{Nil}{\mathbf{Nil}}
\opr{Flag}{\mathbf{Flag}}
\def\vl#1{[#1]} 

\def\brk{\rbr}

\usepackage[colorlinks,allcolors=blue]{hyperref}

\def\note#1{}
\excludeversion{notes}

\begin{document}
\title{Counting twisted Higgs bundles}
\author{Sergey Mozgovoy}
\author{Ronan O'Gorman}

\email{mozgovoy@maths.tcd.ie}
\address{School of Mathematics, Trinity College Dublin, Ireland
\newline\indent
Hamilton Mathematics Institute, Ireland}

\begin{abstract}
We count invariants of the moduli spaces of twisted Higgs bundles on a smooth projective curve.
\end{abstract}

\maketitle
\section{Introduction}
Let $X$ be a smooth projective curve of genus $g$ defined over a finite field $\bF_q$.
Let $L$ be a line bundle of degree $\ell$ over $X$ and let $\hsp_L(r,d)$ be the moduli space of semistable $L$-twisted Higgs bundles over $X$.
It parametrizes pairs $(E,\vi)$, where $E$ is a vector bundle of rank $r$ and degree~$d$ over $X$ and $\vi:E\to E\ts L$ is a homomorphism.
A formula for the computation of the number of points of $\hsp_L(r,d)$ for coprime $r,d$ was conjectured in \cite{mozgovoy_solutions} 
and is proved in this note.

The above conjecture was obtained as a solution of a recursive formula, called an ADHM recursion, conjectured by Chuang, Diaconescu, and Pan \cite{chuang_motivic}.
The ADHM recursion was itself based on a conjectural wall-crossing formula for the refined Donaldson-Thomas invariants on a noncompact 3CY variety $Y=L\oplus(\om_X\ts L\inv)$, where $\om_X$ is the canonical bundle of $X$,
as well as a conjectural formula for the asymptotic ADHM invariants.
The latter invariants can be interpreted as Pandharipande-Thomas invariants of $Y$ \cite{pandharipande_curve}.
The formula counting them was derived in \cite{chuang_motivic} by string theoretic methods, hence remains conjectural from the mathematical point of view.

On the other hand, the formula for $\hsp_L(r,d)$ conjectured in \cite{mozgovoy_solutions} can be considered as a generalization of the conjecture by Hausel and Rodriguez-Villegas \cite{hausel_mixed} in the case of usual Higgs bundles, where the twisting line bundle $L$ is equal to $\om_X$.
A breakthrough for the counting of usual Higgs bundles was made by Schiffmann \cite{schiffmann_indecomposable} who proved an explicit, albeit rather complicated formula for these invariants, quite different from the conjecture of~\cite{hausel_mixed}.
An equivalence between these formulas was proved recently
by purely combinatorial methods
in a brilliant series of papers by Mellit \cite{mellit_integrality,mellit_poincare,mellit_poincarea}.

Results on the invariants of moduli spaces of Higgs bundles
for small rank and degree were obtained in \cite{hausel_mirrora,hausel_mixed,hitchin_self-duality,gothen_betti,garcia-prada_motives,rayan_co,chuang_motivic}.
The conjecture of Hausel and Rodriguez-Villegas was proved for the $y$-genus in \cite{garcia-prada_y}.
An alternative general formula for twisted Higgs bundles on $\bP^1$ -- in terms of quiver representations -- was obtained in~\cite{mozgovoy_higgs}.
Other interesting results related to counting of Higgs bundles can be found in 
\cite{chaudouard_sura,chaudouard_sur,
dobrovolska_moduli,fedorov_motivic}.

In this paper we will apply Mellit's methods in order to prove a formula for general $L$-twisted Higgs bundles.
This task will be rather straightforward as Schiffmann's computation was generalized earlier for twisted Higgs bundles in \msc.
More precisely, let $\hst_L(r,d)$ be the moduli stack of semistable $L$-twisted Higgs bundles over $X$.
Given a finite type algebraic stack $\cX$ over $\bF_q$, define its volume (see \S\ref{sec:volume} for more details on volumes)
\begin{equation}
[\cX]=(\#\cX(\bF_{q^n}))_{n\ge1},\qquad
\#\cX(\bF_{q^n})
=\sum_{x\in\cX(\bF_{q^n})\qt\sim}\frac1{\#\Aut(x)}.
\end{equation}
Define 
(integral) Donaldson-Thomas invariants $\idt_{r,d}$ using the plethystic logarithm (see \S\ref{sec:la-rings})
\begin{equation}
\sum_{d/r=\ta}\idt_{r,d}T^rz^d
=(q-1)\Log\rbr{
\sum_{d/r=\ta}(-q^\oh)^{-\ell r^2}[\hst_L(r,d)]T^rz^d
},\qquad\ta\in\bQ.
\end{equation}
Note that if $r,d$ are coprime, then every $E\in\hst_L(r,d)$ is stable and $\End(E)=\bF_q$ (see Remark~\ref{stable}). Therefore
\begin{equation}
\frac{\vl{\hsp_L(r,d)}}{q-1}
=[\hst_L(r,d)]=(-q^\oh)^{\ell r^2}\frac{\idt_{r,d}}{q-1},
\end{equation}
hence we can recover $\vl{\hsp_L(r,d)}$ from $\idt_{r,d}$.
\begin{notes}
If $r,d$ are coprime and $E$ is semistable of type $(r,d)$, then it is absolutely stable, hence its automorphism group is $\bF_q^*$.
Indeed, let $A=\End(E)$ and $L=k(E)=A/J(A)$ be its splitting field.
If $L\ne K=\bF_q$, then $E\ts_{K}L$ is a direct sum of stable objects of the same type, hence $(r,d)$ is divisible, a contradiction.
\end{notes}
Consider the zeta function of the curve $X$
$$Z_X(t)
=\exp\rbr{\sum_{n\ge1}\frac{\#X(\bF_{q^n})}nt^n}
=\frac{\prod_{i=1}^g(1-\al_it)(1-\al_i\inv qt)}{(1-t)(1-qt)},
$$
where $\al_i$ are the Weil numbers of $X$ (see \S\ref{sec:volume}).
The following result was conjectured in~\cite{mozgovoy_solutions}
(\cf \S\ref{sec:alt}).
We formulate it in the case $\deg L>2g-2$
(see Remark \ref{rem:canonical} for the case $L=\om_X$).


\begin{theorem}[\cf Theorem \ref{th:main proof}]
\label{th:main intro}
Assume that $p=\ell-(2g-2)>0$.
Given a partition $\la$ and a box $s\in\la$, let $a(s)$ and $l(s)$ denote its arm and leg lengths respectively (see \S{\rm\ref{prelim}}).
Define
\begin{equation}
\cdt(T,z)=
\sum_{\la}T^{\n\la}\prod_{s\in\la}
(-q^{a(s)}z^{l(s)})^p\frac
{\prod_{i=1}^g
(q^{a(s)}-\al_i\inv z^{l(s)+1})(q^{a(s)+1}-\al_i z^{l(s)})}
{(q^{a(s)}-z^{l(s)+1})(q^{a(s)+1}-z^{l(s)})},
\end{equation}
\begin{equation}
\sum_{r\ge1}\cidt_r(z)T^r=(q-1)(1-z)\Log\cdt(T,z).
\end{equation}
Then
$\cidt_r(z)\in\bZ[q,z,\al_1^{\pm1},\dots,\al_g^{\pm1}]$ and
$\idt_{r,d}=q^{pr/2}\cidt_r(1)$ for all $d\in\bZ$.
In particular, if $r,d$ are coprime, then
\begin{equation}
\vl{\hsp_L(r,d)}
=(-1)^{pr}q^{(g-1)r^2+p\binom{r+1}2}\cidt_r(1).
\end{equation}
\end{theorem}

\begin{notes}
Tests show that $\cidt_r(1,q)$ has degree $(g-1)r^2+p\binom r2+1$.
This means that the polynomial of the moduli space has degree $2(g-1)r^2+pr^2+1$.
For $L=\om_X$ one has $\idt_{r,d}=q^{pr/2+1}\cidt_r(1,q)$, hence the dimension of the moduli space is $2(g-1)r^2+2$.
\end{notes}

\medskip
\section{Preliminaries}
\label{prelim}

\subsection{Partitions}
A partition is a sequence of integers $\la=(\la_1\ge\la_2\ge\dots)$ such that $\la_n=0$ for $n\gg0$.
We define its length $l(\la)=\#\sets{i}{\la_i\ne0}$ and its weight $\n\la=\sum_{i}\la_i$.
Define its Young diagram (also denoted by $\la$)
\begin{equation}
d(\la)=\sets{(i,j)\in\bZ^2}{i\ge1,1\le j\le\la_i}.
\end{equation}
An element $s=(i,j)\in\la$ is called a box of the Young diagram located at the $i$-th row and $j$-th column.
Define the conjugate partition $\la'$ with $\la'_j$ equal the number of boxes in the $j$-th column of $\la$.
Given a box $s=(i,j)\in\la$, define its arm and leg
lengths respectively
\begin{equation}
\label{eq:arm-leg}
a(s)=\la_i-j,\qquad l(s)=\la'_j-i.
\end{equation}
Define the hook length $h(s)=a(s)+l(s)+1$.

\begin{figure}[h]
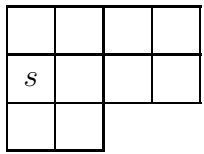

\begin{ytableau}
{}&&&\\
s&&&\\
&
\end{ytableau}
\caption{Young diagram for $\la=(4,4,2)$.
Here $\la'=(3,3,2,2)$, $s=(2,1)$, $a(s)=3$, $l(s)=1$, $h(s)=5$.}
\end{figure}

Define
\begin{equation}
n(\la)=\sum_{s\in\la}l(s)
=\sum_{i\ge1}\binom{\la'_i}2
=\sum_{i\ge1}(i-1)\la_i,
\end{equation}
\begin{equation}
\label{eq:form vs n}
\ang{\la,\la}=\sum_{i\ge1}(\la_i')^2=2n(\la)+\n\la.
\end{equation}


Define
\eq{\label{N}
N_\la(u,q,t)
=
\prod_{s\in\la}(q^{a(s)}-ut^{l(s)+1})(q^{a(s)+1}-u\inv t^{l(s)}).}
One can show that
\begin{equation}
\label{eq:N for conj}
N_{\la}(u,q,t)=N_{\la'}(u,t,q).
\end{equation}

\begin{notes}
\begin{multline*}
N_{\la'}(u,z,q)
=\prod_{s\in\la}(z^l-uq^{a+1})(z^{l+1}-u\inv q^a)\\
=\prod_{s\in\la}(q^{a+1}-u\inv z^l)(q^a-uz^{l+1})
=N_\la(u,q,z).
\end{multline*}
\end{notes}

\subsection{\tpdf{\la}{Lambda}-rings and symmetric functions}
\label{sec:la-rings}
For simplicity we will introduce only \la-rings without \bZ-torsion.
To make things even simpler we can assume that our rings are algebras over~\bQ.
The reason is that in this case the axioms of a \la-ring can be formulated in terms of Adams operations.

Define the graded ring of symmetric polynomials
$$\La_n=\bZ[x_1,\dots,x_n]^{S_n},$$
where $\deg x_i=1$.
Define the ring of symmetric functions
$\La=\ilim\La_n,$
where the limit is taken in the category of graded rings.
For any commutative ring $R$, define $\La_R=\La\ts_\bZ R$.
As in \mac, define generators of $\La$ (complete symmetric and elementary symmetric functions)
$$
h_n=\sum_{i_1\le\dots\le i_n}x_{i_1}\dots x_{i_n},\qquad
e_n=\sum_{i_1<\dots<i_n}x_{i_1}\dots x_{i_n},\qquad
$$
and generators of $\La_\bQ$ (power sums)
$$p_n=\sum_{i}x_i^n.$$
The elements $h_n,e_n,p_n$ have degree $n$.
We also define $h_0=e_0=p_0=1$ for convenience.
For any partition \la of length $\le n$,
define monomial symmetric polynomials
$m_\la=\sum x^\al\in\La_n$, where the sum runs over all distinct permutations $\al=(\al_1,\dots,\al_n)$ of $(\la_1,\dots,\la_n)$.
They induce monomial symmetric functions $m_\la\in\La$ which form a \bZ-basis of $\La$.

A \la-ring $R$ is a commutative ring equipped with a pairing, called plethysm,
$$\La\xx R\to R,\qquad (f,a)\mto f\circ a=f[a],$$
such that with $\psi_n=p_n[-]:R\to R$, called Adams operations, we have
\begin{enumerate}
\item The map $\La\to R,\, f\mto f[a]$, is a ring homomorphism, for all $a\in R$.
\item $\psi_1:R\to R$ is an identity map.
\item The map $\psi_n:R\to R$ is a ring homomorphism,
for all $n\ge1$.
\item $\psi_m\psi_n=\psi_{mn}$, for all $m,n\ge1$.
\end{enumerate}

\begin{remark}\br
\begin{enumerate}
\item
The first axiom implies that it is enough to specify just Adams operations $\psi_n$ or \si-operations $\si_n=h_n[-]$ or \la-operations $\la_n=e_n[-]$.
It also implies that $1[a]=1$, for all $a\in R$.

\item
Usually we equip algebras of the form 
$\bQ[x_1,\dots,x_k]$, $\bQ(x_1,\dots,x_k)$, $\bQ\pser{x_1,\dots,x_k}$
with a \la-ring structure by the formula
$$p_n[f(x_1,\dots,x_k)]=f(x_1^n,\dots,x_k^n).$$

\item
The ring \La can be itself equipped with a \la-ring structure using the same formula
$$
p_m[f]=f(x_1^m,x_2^m,\dots),\qquad f\in\La.
$$
In particular $p_m[p_n]=p_{mn}$.
\item
If $R$ is a \la-ring, then $f\circ(g\circ a)=(f\circ g)\circ a$ for all $f,g\in\La$ and $a\in R$.
\end{enumerate}
\end{remark}

The ring \La can be considered as a free \la-ring with one generator in the following sense.
Consider the category $\lrg$ of \la-rings (with morphisms that respect plethystic operations).
The forgetful functor $F:\lrg\to\Set$
has a left adjoint
$$\Sym:\Set\to\lrg.$$
Given a finite set $\set{X_1,\dots,X_n}$, we denote 
\Sym\set{X_1,\dots,X_n} by $\Sym[X_1,\dots,X_n]$.
Then, for a one-point set $\set{X}$, there is an isomorphism of \la-rings
$$\Sym[X]\iso\La,\qquad X\mto p_1.$$
We will usually identify \La and $\Sym[X]$ using this isomorphism.

Define a filtered \la-ring $R$ to be a \la-ring equipped with a filtration $R=F^0\sps F^1\sps\dots$ such that $F^iF^j\sbs F^{i+j}$ and $\psi_n(F^i)\sbs F^{ni}$.
It is called complete if the natural homomorphism $R\to\ilim R/F^i$ is an isomorphism.
For example, the ring \La is graded, where $\deg h_n=n$.
Hence we have a decomposition $\La=\bop_{k\ge0}\La^k$ 
into graded components.
We equip $\La$ with the filtration $F^k\La=\bop_{i\ge k}\La^i$ and define the completion 
\begin{equation}
\hat\La=\ilim\La/F^k\La\iso\bZ\pser{h_1,h_2,\dots}.
\end{equation}
This ring can be considered as a free complete \la-ring with one generator.
One can see that if $R$ is a complete \la-ring then the plethystic pairing extends to
$$\hat\La\xx F^1R\to R.$$
In particular, the element
\begin{equation}
\Exp[X]=\sum_{n\ge0}h_n[X]=\exp\rbr{\sum_{n\ge1}\frac{p_n[X]}n}
=\prod_{i\ge1}\frac1{1-x_i}\in\hat\La,
\end{equation}
called a plethystic exponential, induces a map $\Exp:F^1R\to 1+F^1R$ which satisfies
\begin{equation}
\Exp[a+b]=\Exp[a]\Exp[b].
\end{equation}
This map has an inverse, called a plethystic logarithm,
\begin{equation}
\Log:1+F^1R\to F^1R,\qquad
\Log[1+a]=\sum_{n\ge1}\frac{\mu(n)}{n}p_n[\log(1+a)].
\end{equation}

\subsection{Modified Macdonald polynomials}
For an introduction to modified Macdonald polynomials see \gar or \cite{mellit_poincarea}.
Let $\cP_n$ denote the set of partitions \la with $\n\la=n$.
Define the natural partial order on $\cP_n$ by
$$\la\le\mu\iff \sum_{i=1}^k\la_i\le\sum_{i=1}^k\mu_i
\quad \forall k\ge1.$$
One can show that $\la\le\mu\iff\mu'\le\la'$ \mac[1.1.11].
Let $\La^{\le\la}\sbs\La$ be the subspace spanned by monomial symmetric functions $m_\mu\in\La$ with $\mu\le\la$.

Let $F=\bQ(q,t)$ and $\La_F=\La\ts_\bZ F$.
For any symmetric function $f\in\La_F$, we will sometimes denote $f[X]$ by $f[X;q,t]$ to indicate dependence on $q,t$.
Let $P_\la[X;q,t]\in\La_F$ be Macdonald polynomials \mac[\S6].
Define modified Macdonald polynomials $\mm_\la[X;q,t]\in\La_F$ \gar[I.8--I.11]
\begin{equation}
\mm_\la[X;q,t]
=H_\la\sbr{X;q,t\inv}\cdot t^{n(\la)},\qquad
H_\la[X]=P_\la\sbr{\frac X{1-t}}
\cdot \prod_{s\in\la}(1-q^{a}t^{l+1}).
\end{equation}
Alternatively, one can uniquely determine $\mm_\la[X;q,t]\in\La_F$ by the properties
\begin{enumerate}
\item $\mm_\la[(1-t)X]\in\La^{\le\la}_F$.
\item Cauchy identity:
\begin{equation*}
\sum_\la
\frac{\mm_\la[X]\mm_\la[Y]}{\prod_{s\in\la}(q^a-t^{l+1})(q^{a+1}-t^l)}
=\Exp\sbr{\frac{XY}{(q-1)(1-t)}}.
\end{equation*}
\end{enumerate}
We have by \gar[Cor.~2.1] (see also \mac[6.6.17])
\begin{equation}
\mm_\la[1-u;q,t]=\prod_{s\in\la}(1-q^{a'(s)}t^{l'(s)}u),
\end{equation}
where $a'(s)=j-1$, $l'(s)=i-1$ for $s=(i,j)\in\la$.
This implies $\mm_\la[1;q,t]=1$.
The symmetric function $\mm_\la$ has degree $\n\la$, hence, applying it to $z\in F[z]$, we obtain
\begin{equation}
\mm_\la[z;q,t]=z^{\n\la}.
\end{equation}
Finally, we have by \gar[Cor.2.2]
\begin{equation}
\mm_\la[X;q,t]=\mm_{\la'}[X;t,q].
\end{equation}

\subsection{Volume ring}
\label{sec:volume}
Following \cite{mozgovoy_poincare},
we will introduce 
in this section a \la-ring which is an analogue of the Grothendieck ring of algebraic varieties or the ring of motives.
We define it to be the ring $\cV=\prod_{n\ge1}\bQ$ with Adams operations
\begin{equation}
\psi_m(a)=(a_{mn})_{n\ge1},\qquad a=(a_n)_{n\ge1}\in \cV,
\end{equation}
and call it the volume ring or the ring of counting sequences \cite{mozgovoy_poincare}. 

Given an algebraic variety $X$ over a finite field $\bF_q$, define its volume
\begin{equation}
[X]=(\# X(\bF_{q^n}))_{n\ge1}\in \cV.
\end{equation}
More generally, given a finite type algebraic stack $\cX$ over $\bF_{q}$, we define its volume
\begin{equation}
[\cX]=(\# \cX(\bF_{q^n}))_{n\ge1}\in \cV,
\end{equation}
where we define, for a finite groupoid $\cG=\cX(\bF_{q^n})$,
\begin{equation}
\#\cG=\sum_{x\in\cG/\sim}\frac1{\#\Aut(x)}.
\end{equation}

Next, let us fix a projective curve $X$ over the field $\bF_q$ and consider its zeta function
\begin{gather}
Z_X(t)
=\exp\rbr{\sum_{n\ge1}\frac{\#X(\bF_{q^n})}nt^n}
=\frac{\prod_{i=1}^g(1-\wr_it)(1-\wr_i\inv qt)} {(1-t)(1-qt)},\\
\label{eq:X}
\#X(\bF_{q^n})
=1+q^n-\sum_{i=1}^g\wr_i^n-q^n\sum_{i=1}^g\wr_i^{-n}
\qquad\forall n\ge1.
\end{gather}

Consider the algebra
\begin{equation}
R_g=\bQ[\bq^{\pm1},\vr_1^{\pm1},\dots\vr_g^{\pm1}
,(\bq^n-1)\inv\rcol n\ge1],
\end{equation}
equipped with the usual \la-ring structure
$$\psi_n(f)=f(\bq^n,\vr_1^n,\dots,\vr_g^n)\qquad\forall f\in R_g.$$
Consider an algebra homomorphism
$$\si:R_g\to\bC,\qquad \bq\mto q,\ \vr_i\mto\wr_i,$$
and a \la-ring homomorphism
\begin{equation}
\bar\si:R_g\to\cV_\bC=\prod_{n\ge1}\bC,\qquad f\mto(\si(\psi_n(f)))_{n\ge1}.
\end{equation}
It restricts to an (injective) \la-ring homomorphism
\begin{equation}
\bar\si:R_g^{S_g\ltimes S_2^g}\to\cV,
\end{equation}
where $S_g$ permutes variables $\vr_i$ 
and the $i$-th copy of $S_2$ permutes $\vr_i$ and $\bq\vr_i\inv$.
Given elements $a\in\cV$ and $f\in R_g$, we will write $a=f$ if $a=\bar\si(f)$.
All equalities in this paper should be understood in this sense.
For example $[\bA^1]=(q^n)_{n\ge1}=\bq$
and \eqref{eq:X} implies
$$[X]=1+\bq-\sum_{i=1}^g\vr_i-\bq\sum_{i=1}^g\vr_i\inv.$$
In what follows we will write $q$ and $\wr_i$ instead of $\bq$ and $\vr_i$ respectively, hoping it will not lead to confusion.

\section{Positive Higgs bundles}
In this section we will review the formula from \msc
counting positive Higgs bundles.
Then we will simplify it using an approach from \mel.
Let $X$ be a smooth projective curve of genus $g$ over a field \bk and let $L$ be a line bundle of degree $\ell$ over~$X$.
Given a coherent sheaf $E\in \Coh X$, we define its slope $\mu(E)=\deg E/\rk E$ and we call $E$ semistable if $\mu(F)\le\mu(E)$ for all $F\sbs E$.

\begin{remark}
\label{stable}
We call $E$ stable if $\mu(F)<\mu(E)$ for all proper $F\sbs E$.
In this case $K=\End(E)$ is a finite-dimensional division algebra over \bk by Schur's lemma.
In particular, $K=\bk$ if \bk is algebraically closed.
If $\rk E,\deg E$ are coprime and $E$ is semistable, then $E$ is automatically stable.
If, moreover, $\bk$ is a finite field, then $K=\bk$.
Indeed, $K$ is a finite (Galois) field extension of \bk by
Wedderburn's little theorem.
We can decompose $E_K=E\ts_\bk K$ over $X_K=X\xx_{\Spec\bk}\Spec K$ as a direct sum $\bop_{\si\in\Gal(K/\bk)}F^\si$, where $F^\si$ have the same rank and degree \cite{mozgovoy_poincare}.
But this would imply that $\rk E,\deg E$ are not coprime if $[K:\bk]>1$.
\end{remark}

Every coherent sheaf $E\in\Coh X$ has a unique filtration, called a Harder-Narasimhan filtration,
$$0=E_0\sbs E_1\sbs\dots\sbs E_n=E$$
such that $E_i/E_{i-1}$ are semistable and $\mu(E_1/E_0)>\dots>\mu(E_n/E_{n-1})$.
We will say that $E$ is positive if $\mu(E_n/E_{n-1})\ge0$.
Equivalenly, for any semistable sheaf $F$ with $\mu(F)<0$, we have $\Hom(E,F)=0$.

Recall that an $L$-twisted Higgs sheaf is a pair $(E,\vi)$, where $E$ is a coherent sheaf over~$X$ and $\vi:E\to E\ts L$ is a homomorphism.
We will say that $(E,\vi)$ is positive if $E$ is positive.
Let $\hgc_L(X)$ be the category of $L$-twisted Higgs sheaves and $\hgc_L^+(X)$ be the category of positive $L$-twisted Higgs sheaves.
We will say that $(E,\vi)\in\hgc_L(X)$ is semistable if $\mu(F)\le\mu(E)$ for every $(F,\vi')\sbs(E,\vi)$.


Let $\hst^\sst_L(r,d)$ 
denote the stack of semistable Higgs bundles 
and $\hst^+_L(r,d)$ denote the stack of positive Higgs bundles (not necessarily semistable) having rank $r$ and degree $d$.
Assuming that $\bk$ is a finite field~$\bF_q$, we define
(exponential) DT invariants
\begin{equation}
\dt_{r,d}
=(-q^\oh)^{-\ell r^2}
[\hst^\sst_L(r,d)]
\end{equation}
and define (integral) DT invariants by the formula
\begin{equation}
\sum_{d/r=\ta}\idt_{r,d}T^rz^d
=(q-1)\Log
\rbr{\sum_{d/r=\ta}\dt_{r,d}T^rz^d}
,\qquad \ta\in\bQ,
\end{equation}
On the other hand, consider the series
\begin{equation}
\pdt(T,z)=
\sum_{r,d} (-q^\oh)^{-\ell r^2}[\hst^+_L(r,d)]
T^rz^d
\end{equation}
and define positive (integral) DT invariants by the formula
\begin{equation}
\label{pdt}
\sum_{r,d}\pidt_{r,d}T^rz^d
=(q-1)\Log\pdt(T,z).
\end{equation}

The following result was proved in \msc:


\begin{theorem}
\label{th:compar}
For every $r\ge1$, we have
\begin{enumerate}
\item $\dt_{r,d+r}=\dt_{r,d}$.
\item $\idt_{r,d+r}=\idt_{r,d}$.
\item $\idt_{r,d}=\pidt_{r,d}$ for $d\gg0$.
\end{enumerate}
\end{theorem}

The last result implies that it is enough to find the positive DT invariants $\pidt_{r,d}$ in order to determine the usual DT invariants $\idt_{r,d}$.
The following explicit formula for the series
$\pdt(T,z)$ was proved in \msc (although the power of $z$ was missing there).

\note{See \S\ref{correct MS} for the correct MS formula}
\begin{theorem}
Assuming that $p=\ell-(2g-2)>0$, we have
$$\pdt(T,z)
=\sum_\la
(-q^\oh)^{\ell\ang{\la,\la}}z^{pn(\la')}J_\la(z)H_\la(z)
T^{\n\la},
$$
where the sum runs over all partitions $\la$ and
$J_\la(z),H_\la(z)$ are certain complicated expressions  defined in \msc.
\end{theorem}

\note{See \S\ref{nt:MS details} for some details.}

The following simplification of the above expression was obtained in \mel[Prop.~3.1].

\begin{proposition}
For every partition \la of length $n$ define
\begin{multline}
\label{f def}
f(z_1,\dots,z_n;q,\bar\al)=\prod_{i=1}^n\prod_{k=1}^g\frac{1-\al_k\inv}{1-\al_k\inv z_i}\\
\xx\sum_{\si\in S_n}\si\rbr{
\prod_{i>j}\rbr{\frac1{1-z_i/z_j}
\prod_{k=1}^g\frac{1-\al_k\inv z_i/z_j}{1-q\al_k\inv z_i/z_j}}
\prod_{i>j+1}(1-qz_i/z_j)
\prod_{i\ge2}(1-z_i)
},
\end{multline}
\begin{equation}
\label{eq:f_la1}
f_\la=f(z_1,\dots,z_n;q,\bar\al),\qquad z_i=q^{i-n}z^{\la_i},\ i=1,\dots,n,
\end{equation}
where $\bar\al=(\al_1,\dots,\al_g)$.
Then (see \eqref{N} for the definition of $N_\la$)
\begin{equation}
q^{(g-1)\ang{\la,\la}}J_\la(z)H_\la(z)
=\frac{\prod_{i=1}^gN_\la(\al_i\inv,z,q)}{N_\la(1,z,q)}
f_\la.
\end{equation}
\end{proposition}

\note{See \S\ref{nt:further form for Om_la} for additional formulas.}

The last two results imply

\begin{corollary}
\label{cor:ms-mel}
Assuming that $p=\ell-(2g-2)>0$, we have
\begin{equation}
\label{eq:cor:ms-mel}
\pdt(q^{-p/2}T,z)
=\sum_\la
\rbr{(-1)^{\n\la}q^{n(\la')}z^{n(\la)}}^p
\frac{\prod_{i=1}^gN_\la(\al_i\inv,q,z)}{N_\la(1,q,z)}
f_{\la'}\cdot
T^{\n\la}
.
\end{equation}
\end{corollary}
\begin{proof}
Using the fact that $\ang{\la,\la}=2n(\la)+\n\la$ (see \eqref{eq:form vs n}),
we obtain
$$\pdt(T,z)
=\sum_\la
\rbr{(-1)^{\n\la}q^{n(\la)}z^{n(\la')}}^p
\frac{\prod_{i=1}^gN_\la(\al_i\inv,z,q)}{N_\la(1,z,q)}
f_\la\cdot
(q^{p/2}T)^{\n\la}.$$
Now we sum over conjugate partitions and apply \eqref{eq:N for conj}.
\end{proof}

\begin{lemma}
\label{lm:laurent}
We have 
$$f \in
\bQ[z_1^{\pm1},\dots,z_n^{\pm1};q^{\pm1}]\pser{\al_1\inv,\dots,\al_g\inv}.$$
\end{lemma}
\begin{proof}
The factors $(1 - z_i/z_j)$ disappear from the denominator of $f$ when we sum over $S_n$, so looking at the remaining factors we see that
$$f(z_1,\dots,z_n)\cdot\prod_{k = 1}^g\brk{\prod_{i = 1}^n\brk{1 - \alpha_k\inv z_i} \prod_{i \ne j} \brk{1 - q\alpha_k\inv z_i/z_j} }$$
is a Laurent polynomial.
The result follows on observing that every factor in the brackets is invertible in
$\bQ[z_1^{\pm1},\dots,z_n^{\pm1};q^{\pm1}]\pser{\al_1\inv,\dots,\al_g\inv}$.
\end{proof}

\begin{proposition}[see {\mel[\S4.2]}]
\label{prop:inductive f}
We have
$$f(1,z_1,\dots,z_n)=f(qz_1,\dots,qz_n).$$
\end{proposition}

\section{Main result}

\subsection{Admissibility}
\note{See \S\ref{nt:technical}
for the original formulation of the technical result \mel[Lemma 5.1].
Here we change the roles of $q,t$.}

Let $R$ be a \la-ring flat over $\bQ(q)[t^{\pm1}]$ and let $R^*=R\ts_{\bQ(q)[t^{\pm1}]}\bQ(q,t)$.
We will say that $F\in R^*$ is admissible if $(1-t)\Log F$ 
is contained in $R$
(usually $R$ will be clear from the context).
In view of Proposition \ref{prop:inductive f}, we introduce the following concept.

\begin{notes}
We could also consider $R=\bQ[q^{\pm1},t^{\pm1}]$ and admissibility \wrt $S=(q-1)(1-t)$.
But in the proof of [Mellit, Lemma1 5.1] he uses some rational functions in $q$ ($Z_\mu$ on p.15; he uses variable $t$ instead of $q$).
\end{notes}

\begin{definition}
Let $q\in R$ be an invertible element.
For every $n\ge0$, consider rings $\bar\La_n=R[z_1^{\pm1},\dots,z_n^{\pm1}]^{S_n}$
and ring homomorphisms
$$\pi_n:\bar\La_{n+1}\to\bar\La_n,\qquad 
(\pi_n f)(z_1,\dots,z_{n})=f(1,q\inv z_1,\dots,q\inv z_n).$$
Define a $q$-twisted symmetric function $f=(f_n)_{n\ge0}$ 
to be an element of $\bar\La=\ilim\bar\La_n$.
\end{definition}

\begin{notes}
Let $\tl f_n(z_1,\dots,z_n)=f_n(t^{-n}z_1,\dots,t^{-n}z_n)$.
Then 
$$\tl f_n(z_1,\dots,z_n)=
\tl f_{n+1}(t^{n+1},z_1,\dots,z_n).$$
\end{notes}

Given a $q$-twisted symmetric function $f$, define for any partition \la (\cf \eqref{eq:f_la1})
\begin{equation}
f_\la=f_n(z_1,\dots,z_n),\qquad z_i=q^{i-n}t^{\la_i},
\ n\ge l(\la).
\end{equation}
Note that this expression is independent of the choice of $n\ge l(\la)$.

\begin{remark}
The following result is a reformulation of \mel[Lemma 5.1].
Here we exchange the roles of $q,t$ and use conjugate partitions.
We also add an invertible factor $(q-1)$.
\end{remark}

\begin{theorem}
\label{th:tech}
Let $f(u)=\sum_{i\ge0}f^{(i)}u^i\in\bar\La\pser u$ be a power series with $f^{(0)}=1$ and let
$$\dt[X;u]=\sum_\la c_{\la}\wtl H_{\la}[X;q,t] f_{\la'}(u),\qquad
\idt[X;u]=(q-1)(1-t)\Log \dt[X;u],$$
where $c_\la\in R^*$ and $c_\es=1$.
If $\dt[X;0]$ is admissible, 
then $\idt[X;u]-\idt[X;0]$ has coefficients in $(t-1)R$.
In particular, $\idt[X;u]$ is independent of $u$ at $t=1$.
\end{theorem}

\subsection{Proof of the main theorem}
\note{See \S\ref{nt:other cdt} for other formulas of $\cdt$}

Now we are ready to prove Theorem \ref{th:main intro} from the introduction.
For this section we will use the variable $t$ in place of $z$ as it is customary in the theory of orthogonal symmetric polynomials.

\begin{theorem}[\cf Theorem \ref{th:main intro}]
\label{th:main proof}
Assume that $p=\ell-(2g-2)>0$.
Define
(see \eqref{N} for the definition of $N_\la$)
\begin{equation}
\label{eq:th:cdt}
\cdt(T,q,t)=
\sum_{\la}
\rbr{(-1)^{\n\la}q^{n(\la')}t^{n(\la)}}^p
\frac{\prod_{i=1}^gN_\la(\al_i\inv,q,t)}{N_\la(1,q,t)}
T^{\n\la},
\end{equation}
\begin{equation}
\cidt(T,q,t)
=\sum_{r\ge1}\cidt_r(q,t)T^r=(q-1)(1-t)\Log\cdt(T,q,t).
\end{equation}
Then $\cidt_r(q,t)\in\bZ[q,t,\al_1^{\pm1},\dots,\al_g^{\pm1}]$ and
$$\idt_{r,d}=q^{pr/2}\cidt_r(q,1)\qquad \forall d\in\bZ.$$

\end{theorem}
\begin{proof}
According to Theorem \ref{th:compar} it is enough to show that
$\pidt_{r,d}=q^{pr/2}\cidt_r(q,1)$
for $d\gg0$,
where $\pidt_{r,d}$ are determined by \eqref{pdt} and Corollary \ref{cor:ms-mel}:
\begin{equation}
\label{eq:dt pos}
\pdt(q^{-p/2}T,q,t)
=\sum_\la
\rbr{(-1)^{\n\la}q^{n(\la')}t^{n(\la)}}^p
\frac{\prod_{i=1}^gN_\la(\al_i\inv,q,t)}{N_\la(1,q,t)}
f_{\la'}\cdot
T^{\n\la},
\end{equation}
\begin{equation}
\pidt(T,q,t)=\sum_r\pidt_r(q,t)T^r=\sum_{r,d}\pidt_{r,d} T^rt^d=(q-1)\Log\pdt(T,q,t).
\end{equation}

We will compare the series $\pdt(q^{-p/2}T,q,t)$ 
to the series $\cdt(T,q,t)$
using Theorem~\ref{th:tech}
with the ring of Laurent series
\begin{equation}
R=
\bQ(q)[t^{\pm1}]\lser{\al_1\inv,\dots,\al_g\inv}
\end{equation}
and the series 
$\tl f(u)=\sum_{i\ge0}\tl f^{(i)}u^i$
which is a deformation of $f$ \eqref{f def} defined by
$$
\tl f^{(i)}=(\tl f_n^{(i)})_{n\ge0},\qquad
\tl f_n(z_1,\dots,z_n;u)
=\sum_{i\ge0}\tl f_n^{(i)}u^i
=f(z_1,\dots,z_n;q,u\inv\bar\al),$$
where every $\al_i$ is substituted by $u\inv\al_i$.
It follows from Lemma \ref{lm:laurent} that 
$$
\tl f_n\in \bQ[q^{\pm1},\al_1^{\pm1},\dots,\al_g^{\pm1}][z_1^{\pm1},\dots,z_n^{\pm1}]^{S_n}\pser u,$$
hence by Proposition \ref{prop:inductive f} the coefficients $\tl f^{(i)}$ are $q$-twisted symmetric functions over $R$.
It follows from \mel[Theorem 5.2] that $\tl f_n|_{u=0}=1$,
hence $\tl f(0)=1$.

As before, define
$$\tl f_\la(u)=\tl f_n(z_1,\dots,z_n;u),\qquad z_i=q^{i-n}t^{\la_i},\ n\ge l(\la),$$
and consider the series of symmetric functions
\begin{align}
\label{eq:def om}
\dt[X;u,q,t]&=\sum_\la
\rbr{(-1)^{\n\la}q^{n(\la')}t^{n(\la)}}^p
\frac{\prod_{i=1}^gN_\la(\al_i\inv,q,t)}{N_\la(1,q,t)}
\wtl H_\la[X;q,t] \tl f_{\la'}(u),\\
\idt[X;u,q,t]&=(q-1)(1-t)\Log\dt[X;u,q,t].
\end{align}
Then \eqref{eq:th:cdt} and \eqref{eq:dt pos} translate to
\begin{align*}
\dt[T;0,q,t]&=\cdt(T,q,t),
&\idt[T;0,q,t]&=\cidt(T,q,t),\\
\dt[T;1,q,t]&=\pdt(q^{-p/2}T,q,t),
&\idt[T;1,q,t]&=(1-t)\pidt(q^{-p/2}T,q,t).
\end{align*}

In order to apply Theorem~\ref{th:tech} we need to show that
$$\dt[X;0,q,t]=\sum_\la
\rbr{(-1)^{\n\la}q^{n(\la')}t^{n(\la)}}^p
\frac{\prod_{i=1}^gN_\la(\al_i\inv,q,t)}{N_\la(1,q,t)}
\wtl H_\la[X;q,t]
$$
is admissible.
The series 
$$\sum_\la
\frac{\prod_{i=1}^gN_\la(\al_i\inv,q,t)}{N_\la(1,q,t)}
\wtl H_\la[X;q,t]$$
is admissible according to \melc.
The operator $\Na$ defined by 
$$\wtl H_\la\mto(-1)^{\n\la}q^{n(\la')}t^{n(\la)}\wtl H_\la$$
preserves admissibility by \melc[Cor.~6.3].
Therefore the series $\dt[X;0,q,t]$ is also admissible (one actually obtains from \melc that the coefficients of $\idt[X;0,q,t]$ are in
$\bZ[q,t,\al_1^{\pm1},\dots,\al_g^{\pm1}]$, hence the same is true for $\cidt(T,q,t)$).

We conclude from Theorem \ref{th:tech} that
\begin{equation}
\label{eq:difference}
\idt[T;u,q,t]-\idt[T;0,q,t]\in (1-t)R\pser{T,u}.
\end{equation}
By Lemma \ref{lm:laurent} we can consider $\dt[T;u,q,t]$ \eqref{eq:def om}
as a series
with polynomial coefficients in $u$
$$\dt[T;u,q,t]\in\bQ(q,t)[u]\lser{\al_1\inv,\dots,\al_g\inv}\pser T.$$
The same then applies to $\idt[T;u,q,t]$ and we can set $u=1$ in \eqref{eq:difference}.
We obtain
\begin{equation*}
(1-t)\pidt(q^{-p/2}T,q,t)-\cidt(T,q,t)
\in(1-t)R\pser T.
\end{equation*}
This implies that $(1-t)q^{-pr/2}\pidt_r(q,t)-\cidt_r(q,t)=(1-t)g$ for some $g\in R$.
Therefore $$q^{-pr/2}\sum_{d\ge0}\pidt_{r,d}t^d=\frac{\cidt_r(q,t)}{1-t}+g.$$
Comparing the coefficients of the monomials in $\al_1,\dots,\al_g$ and using the fact that $\pidt_{r,d+r}=\pidt_{r,d}$ for $d\gg0$,
we conclude that $q^{-pr/2}\pidt_{r,d}=\cidt_r(q,1)$ for $d\gg0$.
\end{proof}

\begin{notes}
We should have $\cidt_r[t=1]-(1-t)\pidt_r=(1-t)f(t)$ for some $f\in R$.
Therefore the above constant is actually $\cidt_r[t=1]$.
\end{notes}

\begin{remark}
\label{rem:canonical}
Let us also formulate the result in the case $L=\om_X$ (the canonical bundle) for completeness \mel.
In this case we have $\ell=2g-2$ and $p=\ell-(2g-2)=0$.
Define as before
\begin{equation}
\cdt(T,q,t)=
\sum_{\la}
\frac{\prod_{i=1}^gN_\la(\al_i\inv,q,t)}{N_\la(1,q,t)}
T^{\n\la}
\end{equation}
\begin{equation}
\cidt(T,q,t)
=\sum_{r\ge1}\cidt_r(q,t)T^r=(q-1)(1-t)\Log\cdt(T,q,t).
\end{equation}
Using results of \msc and the same proof as before, we obtain the formula for integral Donaldson-Thomas invariants
$\idt_{r,d}=q\cidt_r(q,1)$ (note the additional factor $q$).
These invariants are related to the invariants $A_{r,d}$ counting absolutely indecomposable vector bundles of rank $r$ and degree~$d$ over $X$: $\idt_{r,d}=qA_{r,d}$ \msc.
This implies $A_{r,d}=\cidt_r(q,1)$, as was proved by Mellit \mel.
\end{remark}

\subsection{Alternative formulation}
\label{sec:alt}
The following result was conjectured in \cite[Conj.~3]{mozgovoy_solutions}.

\def\dta{\mathbf H^\circ}
\begin{theorem}
Assume that $p=\ell-(2g-2)>0$.
Consider the series 
$$\cH(T,q,t)=
\sum_{\la}T^{\n\la}
\prod_{s\in\la}
(-t^{a(s)-l(s)}q^{a(s)})^p
t^{(1-g)(2l(s)+1)}Z_X(t^{h(s)}q^{a(s)}),$$
$$\dta(T,q,t)=\sum_{r\ge1}\dta_r(q,t)T^r=(1-t)(1-qt)\Log\cH(T,q,t).$$
Then $\dta_r(q,t)\in\bZ[q,t^{\pm1},\al_1^{\pm1},\dots,\al_g^{\pm1}]$
and $\idt_{r,d}=q^{pr/2}\dta_r(q,1)$.
\end{theorem}
\begin{notes}
We can not substitute $q\mto qt\inv$ and get $Z_X(t^{l+1}q^a)$.
The reason is that we should also substitute the Weil root $\al_i\inv q$ by $\al_i\inv qt\inv$.
\end{notes}
\begin{proof}
Using the substitution $t\mto t\inv$, we obtain
\begin{multline*}
\cH(T,q,t\inv)
=\sum_{\la}T^{\n\la}\prod_{s\in\la}(-t^{l-a}q^a)^p
t^{(g-1)(2l+1)}Z_X(t^{-h}q^a)\\
=\sum_{\la}T^{\n\la}
\prod_{s\in\la}(-t^{l-a}q^a)^p
\frac{\prod_{i=1}^g(t^{l+1}-\al_i t^{-a}q^a)(t^l-\al_i\inv t^{-a-1}q^{a+1})}
{(t^{l+1}-t^{-a}q^a)(t^l-t^{-a-1}q^{a+1})},
\end{multline*}
while
$$t\dta(T,q,t\inv)=(1-t)(t\inv q-1)\Log\cH(T,q,t\inv).$$
Using the substitution $q\mto qt$, we obtain
\begin{multline*}
\cH(T,qt,t\inv)=
\sum_{\la}T^{\n\la}
\prod_{s\in\la}
(-t^{l}q^a)^p
\frac{\prod_{i=1}^g(t^{l+1}-\al_i q^a)(t^l-\al_i\inv q^{a+1})}{(t^{l+1}-q^a)(t^l-q^{a+1})}
\\
=\sum_{\la}T^{\n\la}
\rbr{(-1)^{\n\la}q^{n(\la')}t^{n(\la)}}^p
\frac{\prod_{i=1}^g N_\la(\al_i\inv,q,t)}{N_\la(1,q,t)}.
\end{multline*}
Now the result follows from Theorem \ref{th:main proof}.
\end{proof}


\providecommand{\bysame}{\leavevmode\hbox to3em{\hrulefill}\thinspace}
\providecommand{\href}[2]{#2}

\end{document}